\documentclass[11pt,reqno]{amsart}

\usepackage[T1]{fontenc}
\usepackage[utf8]{inputenc}
\usepackage{lmodern}
\usepackage{microtype}
\usepackage[a4paper,margin=1.15in]{geometry}
\usepackage{amsmath,amssymb,amsthm,mathtools}
\usepackage{mathrsfs}
\usepackage{enumitem}
\usepackage{aliascnt}
\usepackage{xcolor}
\usepackage{hyperref}
\usepackage[nameinlink,capitalize,noabbrev]{cleveref}

\definecolor{linkblue}{RGB}{20,63,110}
\definecolor{citegreen}{RGB}{35,92,67}
\hypersetup{
  colorlinks=true,
  linkcolor=linkblue,
  citecolor=citegreen,
  urlcolor=linkblue,
  pdfauthor={Nobuo Iida},
  pdftitle={On Auroux's Luttinger-Surgery Uniqueness Question},
  pdfsubject={Symplectic topology and four-manifolds},
  pdfkeywords={Luttinger surgery, symplectic four-manifold, symplectic Kodaira dimension, Craighero--Gattazzo surface, rational surface, symplectic blowup, Auroux uniqueness question}
}

\numberwithin{equation}{section}
\allowdisplaybreaks[1]
\newtheorem{theorem}{Theorem}[section]

\newaliascnt{proposition}{theorem}
\newtheorem{proposition}[proposition]{Proposition}
\aliascntresetthe{proposition}

\newaliascnt{lemma}{theorem}

\aliascntresetthe{lemma}

\newaliascnt{corollary}{theorem}
\newtheorem{corollary}[corollary]{Corollary}
\aliascntresetthe{corollary}

\theoremstyle{definition}
\newaliascnt{definition}{theorem}

\aliascntresetthe{definition}

\newaliascnt{question}{theorem}
\newtheorem{question}[question]{Question}
\aliascntresetthe{question}
\crefname{question}{Question}{Questions}

\theoremstyle{remark}
\newaliascnt{remark}{theorem}

\aliascntresetthe{remark}

\newcommand{\R}{\mathbb{R}}
\newcommand{\Z}{\mathbb{Z}}
\newcommand{\CP}{\mathbb{CP}}
\newcommand{\CPbar}{\overline{\mathbb{CP}}}
\newcommand{\PD}{\operatorname{PD}}

\newcommand{\ks}{\kappa^{\mathrm{s}}}

\newcommand{\homeo}{\cong_{\mathrm{homeo}}}

\title[Auroux's Luttinger-surgery question]
{On Auroux's Luttinger-Surgery Uniqueness Question}
\author{Nobuo Iida}
\address{Kavli Institute for the Physics and Mathematics of the Universe (WPI), The University of Tokyo, 5-1-5 Kashiwanoha, Kashiwa, Chiba 277-8583, Japan}
\email{iidanobuo1224@g.ecc.u-tokyo.ac.jp}
\date{August 26, 2026}

\begin{document}

\begin{abstract}
We construct simply connected integral symplectic four-manifolds forming
an exotic pair, with equal values of $c_1^2$, $c_2$,
$c_1\mathbin{\cdot}[\omega]$, and $[\omega]^2$, but with symplectic
Kodaira dimensions $-\infty$ and $2$.  Since symplectic Kodaira dimension
is invariant under Luttinger surgery, the two manifolds are not
Luttinger-surgery equivalent.  This gives a negative answer to Auroux's
uniqueness question.
\end{abstract}

\maketitle

\section{Introduction}

Luttinger surgery is a symplectic cut-and-paste operation along a
Lagrangian torus, introduced by Luttinger \cite{Luttinger1995} and
developed as a tool in symplectic four-manifold topology by
Auroux--Donaldson--Katzarkov \cite{AurouxDonaldsonKatzarkov2003}.  It
preserves the Euler characteristic, the signature, the symplectic
volume, and the pairing of the first Chern class with the symplectic
class.  Its flexibility led Auroux to ask whether these numerical data
are complete for equivalence under Luttinger surgery.

We call a symplectic form $\omega$ integral if $[\omega]$ lies in the
image of the natural map $H^2(X;\Z)\to H^2(X;\R)$.  For a closed
symplectic four-manifold, $c_2(X)$ denotes the Chern number
$\langle c_2(TX,J),[X]\rangle$ for any $\omega$-compatible almost
complex structure $J$; this number equals $\chi(X)$.  Auroux's original
question can then be stated as follows.

\begin{question}[Auroux \cite{Auroux2005Open}]\label{ques:Auroux}
Let $(X_1,\omega_1)$ and $(X_2,\omega_2)$ be closed integral symplectic
four-manifolds such that
\[
 \bigl(c_1(\omega_1)^2,c_2(X_1),
       c_1(\omega_1)\mathbin{\cdot}[\omega_1],[\omega_1]^2\bigr)
 =
 \bigl(c_1(\omega_2)^2,c_2(X_2),
       c_1(\omega_2)\mathbin{\cdot}[\omega_2],[\omega_2]^2\bigr).
\]
Can $(X_2,\omega_2)$ always be obtained from $(X_1,\omega_1)$ by a
finite sequence of Luttinger surgeries?
\end{question}

This is also recorded as the Auroux uniqueness question in
McDuff--Salamon \cite[Problem~11]{McDuffSalamon2017}.  For a closed
symplectic four-manifold, one has
\[
c_2(X)=\chi(X),\qquad
c_1(\omega)^2=2\chi(X)+3\sigma(X),
\]
while
\[
\operatorname{Vol}(X,\omega)
=\frac12[\omega]^2.
\]
Thus the numerical assumptions in Question~\ref{ques:Auroux}
are equivalent to requiring that the two manifolds have the same
Euler characteristic, signature, symplectic volume, and value of
$c_1(\omega)\cdot[\omega]$, as in the formulation of
McDuff--Salamon.
Auroux's uniqueness question is a natural unstable
counterpart of Auroux's stable classification theorem: integral
symplectic four-manifolds with the same four numbers become
symplectomorphic after suitable blowups and stabilizing symplectic sums
\cite{Auroux2005Stable}.

Ho--Li proved that symplectic Kodaira dimension is invariant under
Luttinger surgery and explicitly discussed its relevance to
\cref{ques:Auroux} \cite[Theorem~1.1 and Remark~3.4]{HoLi2012}.  They
also proved the stronger statement that every Luttinger surgery on a
symplectic four-manifold of Kodaira dimension $-\infty$ produces a
symplectomorphic manifold \cite[Theorem~1.2]{HoLi2012}.  Baykur later
obtained negative results for the more structured problem of relating
Lefschetz fibrations or pencils by fibered Luttinger surgeries, while
emphasizing that those results did not settle Auroux's question for the
underlying symplectic manifolds \cite[Remark~4.3]{Baykur2016}.

Ho--Li's theorem supplies the obstruction, but a counterexample to
\cref{ques:Auroux} still requires integral symplectic forms on manifolds
of different symplectic Kodaira dimensions for which both the
symplectic square and the canonical pairing agree.  We carry out this
Diophantine matching explicitly on an exotic pair of rational-surface
type.

Let $S$ denote the Craighero--Gattazzo surface.  This is a simply
connected numerical Godeaux surface with ample canonical bundle and
$K_S^2=1$ \cite{CraigheroGattazzo1994,DolgachevWerner1999,
DolgachevWerner2001,RanaTevelevUrzua2017}.  Set
\[
  X^-:=\CP^2\#10\CPbar^{\,2},
  \qquad
  X^+:=S\#2\CPbar^{\,2}.
\]
The two smooth manifolds are orientation-preservingly homeomorphic but
not diffeomorphic.  Indeed, their symplectic Kodaira dimensions are
$-\infty$ and $2$, so Li's form-independence theorem rules out an
orientation-preserving diffeomorphism
\cite[Theorem~2.6]{Li2006}.  Since both manifolds have signature $-9$,
an orientation-reversing diffeomorphism is impossible
as well.  Thus $X^-$ and $X^+$ form an exotic pair.

Our main result matches all four of Auroux's numerical invariants by
explicit integral cohomology classes represented by symplectic forms.

\begin{theorem}[Main theorem]\label{thm:main}
There exists an integer $m_0\ge2$ such that, for every integer
$m\ge m_0$, there are integral symplectic forms $\omega_m^-$ on $X^-$
and $\omega_m^+$ on $X^+$ with the following properties.
\begin{enumerate}[label=\textup{(\roman*)},leftmargin=2.8em]
  \item Both manifolds are simply connected, and
  \[
      X^-\homeo X^+\homeo \CP^2\#10\CPbar^{\,2}.
  \]
  \item Their numerical invariants agree and are given by
  \begin{equation}\label{eq:main-numerics}
  \begin{aligned}
    c_1(\omega_m^-)^2=c_1(\omega_m^+)^2&=-1,\\
    c_2(X^-)=c_2(X^+)&=13,\\
    c_1(\omega_m^-)\mathbin{\cdot}[\omega_m^-]
      =c_1(\omega_m^+)\mathbin{\cdot}[\omega_m^+]
      &=-(3m+4),\\
    [\omega_m^-]^2=[\omega_m^+]^2&=9m^2+12m+2.
  \end{aligned}
  \end{equation}
  \item Their symplectic Kodaira dimensions satisfy
  \[
      \ks(X^-,\omega_m^-)=-\infty,
      \qquad
      \ks(X^+,\omega_m^+)=2.
  \]
\end{enumerate}
Consequently, $(X^-,\omega_m^-)$ and $(X^+,\omega_m^+)$ are not
Luttinger-surgery equivalent.
\end{theorem}

\begin{corollary}\label{cor:Auroux-negative}
The answer to \cref{ques:Auroux} is negative even for simply connected
symplectic four-manifolds that form an exotic pair.
\end{corollary}

To the best of the author's knowledge, the construction below is the
first explicit negative answer to Auroux's question. 

\section{Background}\label{sec:background}

\subsection{Canonical classes and symplectic Kodaira dimension}

Let $(X,\omega)$ be a closed symplectic four-manifold.  The space of
$\omega$-compatible almost-complex structures is contractible, so
\[
 c_1(\omega):=c_1(TX,J)
\]
is independent of the choice of compatible $J$.  We write
\[
 K_\omega:=-c_1(\omega)
\]
for the symplectic canonical class.  The identities
\begin{equation}\label{eq:chern-topology}
 c_1(\omega)^2=2e(X)+3\sigma(X),
 \qquad
 c_2(X)=e(X).
\end{equation}
Here $c_2(X)=\langle c_2(TX,J),[X]\rangle$.  These identities follow
from the Hirzebruch signature theorem and the Chern--Gauss--Bonnet
theorem.

Recall the definition of symplectic Kodaira dimension
\cite{Li2006}.  If $(X,\omega)$ is symplectically minimal, then
\[
\ks(X,\omega)=
\begin{cases}
-\infty,& K_\omega^2<0\text{ or }K_\omega\cdot[\omega]<0,\\
0,&K_\omega^2=0\text{ and }K_\omega\cdot[\omega]=0,\\
1,&K_\omega^2=0\text{ and }K_\omega\cdot[\omega]>0,\\
2,&K_\omega^2>0\text{ and }K_\omega\cdot[\omega]>0.
\end{cases}
\]
For a nonminimal symplectic four-manifold, $\ks$ is defined to be the
Kodaira dimension of a minimal model.  It is independent of the sequence
of symplectic blowdowns and is invariant under symplectic blowup.  It is
also independent of the chosen symplectic form on a fixed oriented smooth
four-manifold admitting symplectic structures; see
\cite[Theorem~2.6]{Li2006} and the references therein.  Rational and
ruled symplectic four-manifolds have
Kodaira dimension $-\infty$, while a minimal K\"ahler surface of general
type has symplectic Kodaira dimension $2$.

The fact that a minimal model enters the definition is crucial here.
The two forms constructed below have the same values of $K_\omega^2$ and
$K_\omega\cdot[\omega]$ on the nonminimal manifolds themselves, but their
minimal models lie at opposite ends of the Kodaira-dimension scale.

\subsection{Luttinger surgery and the Ho--Li obstruction}

A Luttinger surgery removes a Weinstein neighborhood of a Lagrangian
torus and reglues it by a boundary diffeomorphism of prescribed slope so
that the symplectic form extends over the reattached torus
neighborhood $T^2\times D^2$.  We use
the conventions of \cite{Luttinger1995,AurouxDonaldsonKatzarkov2003}.
Only the following theorem is needed.

\begin{theorem}[Ho--Li]\label{thm:HoLi}
\cite[Theorems~1.1 and~1.2]{HoLi2012}
Suppose that $(\widetilde X,\widetilde\omega)$ is obtained from
$(X,\omega)$ by a Luttinger surgery.  Then
\[
  \ks(\widetilde X,\widetilde\omega)=\ks(X,\omega).
\]
If $\ks(X,\omega)=-\infty$, then
$(\widetilde X,\widetilde\omega)$ is in fact symplectomorphic to
$(X,\omega)$.
\end{theorem}

\begin{corollary}\label{cor:sequence-kappa}
Symplectic Kodaira dimension is invariant under every finite sequence of
Luttinger surgeries.  The same conclusion holds for surgeries performed
along any finite disjoint collection of Lagrangian tori.
\end{corollary}

\begin{proof}
The first assertion follows by induction from \cref{thm:HoLi}.  For a
disjoint collection, order the tori arbitrarily.  Luttinger surgery is
local: after surgery on one torus, the symplectic form is unchanged on
the complement of its chosen neighborhood.  Hence every remaining torus
is still Lagrangian, and the simultaneous operation can be performed
sequentially.
\end{proof}

\subsection{Reduced blowup forms on rational surfaces}

Let
\[
 X_N=\CP^2\#N\CPbar^{\,2}.
\]
Write $H,E_1,\ldots,E_N$ for the standard basis of $H_2(X_N;\Z)$ and,
by abuse of notation, also for the corresponding classes in cohomology
under the intersection-form identification.  Thus
\[
 H^2=1,\qquad E_i^2=-1,\qquad H\cdot E_i=E_i\cdot E_j=0\quad(i\ne j).
\]
A vector
\[
 (a;b_1,\ldots,b_N),\qquad a,b_i>0,
\]
encodes the class
\begin{equation}\label{eq:encoded-class}
 \Omega=aH-\sum_{i=1}^N b_iE_i.
\end{equation}
It is \emph{ordered} if $b_1\ge\cdots\ge b_N$ and \emph{reduced} if
\begin{equation}\label{eq:reduced}
 b_1+b_2+b_3\le a.
\end{equation}

Karshon--Kessler use the convention that the vector
$(a;b_1,\ldots,b_N)$ encodes the cohomology class of a blowup form
$\omega_{\mathrm{KK}}$ when, in the marking above,
\[
 \frac{1}{2\pi}[\omega_{\mathrm{KK}}]
 =aH-\sum_{i=1}^N b_iE_i.
\]
We rescale by setting $\omega=(2\pi)^{-1}\omega_{\mathrm{KK}}$.
Then $[\omega]=aH-\sum_i b_iE_i$.  Positive rescaling preserves the
set of compatible almost-complex structures, and hence does not change
the symplectic canonical class or the symplectic Kodaira dimension.
Thus their criterion takes the following form in our normalization.

\begin{theorem}[Karshon--Kessler]\label{thm:KK}
\cite[Remark~1.8 and Theorem~1.9]{KarshonKessler2017}
Let $N\ge3$, and let
\[
 \Omega=aH-\sum_{i=1}^N b_iE_i,
\]
where $(a;b_1,\ldots,b_N)$ is an ordered reduced vector with positive
entries.  Then there exists a blowup symplectic form $\omega$ on $X_N$
satisfying $[\omega]=\Omega$ if and only if
\[
 \Omega^2
 =a^2-\sum_{i=1}^N b_i^2
 >0.
\]
\end{theorem}

For a blowup form in the marking above, the symplectic canonical class is
\begin{equation}\label{eq:rational-canonical}
 K_\omega=-3H+E_1+\cdots+E_N.
\end{equation}
This follows either by successively applying the canonical-class blowup
formula or by blowing down the standard symplectic exceptional spheres;
see, for example, \cite{McDuff1990,McDuffSalamon2017}.  The broader
structure of the symplectic cone for $b^+=1$ was developed by Li--Liu
\cite{LiLiu2001}; \cref{thm:KK} is the convenient explicit criterion
needed here.

\subsection{The Craighero--Gattazzo surface}

We collect the properties of the complex surface used in the
construction.

\begin{proposition}\label{prop:CG}
There exists a smooth simply connected minimal complex surface $S$ with
\[
 p_g(S)=q(S)=0,
\]
whose canonical line bundle $\mathcal K_S$ is ample and whose canonical
class
\[
 K_S:=c_1(\mathcal K_S)
\]
satisfies $K_S^2=1$.  Moreover, it satisfies
\[
 e(S)=11,\qquad \sigma(S)=-7,
\]
and is orientation-preservingly homeomorphic to
$\CP^2\#8\CPbar^{\,2}$.
\end{proposition}

\begin{proof}
The surface is the Craighero--Gattazzo numerical Godeaux surface
\cite{CraigheroGattazzo1994}.  Dolgachev--Werner established the
canonical polarization and analyzed its fundamental group
\cite{DolgachevWerner1999,DolgachevWerner2001}; the proof that the
surface is simply connected was completed by Rana--Tevelev--Urz\'ua
\cite{RanaTevelevUrzua2017}.

Since $p_g=q=0$, the holomorphic Euler characteristic is
$\chi(\mathcal O_S)=1$.  Noether's formula gives
\[
 e(S)=c_2(S)=12\chi(\mathcal O_S)-K_S^2=11.
\]
The signature formula $K_S^2=2e(S)+3\sigma(S)$ then gives
$\sigma(S)=-7$.  Thus $b_2(S)=9$, $b_2^+(S)=1$, and $b_2^-(S)=8$.
The intersection form is odd because the characteristic class $K_S$ has
odd square.  Hence it is the odd unimodular form $I_{1,8}$.  Both $S$
and $\CP^2\#8\CPbar^{\,2}$ are smooth, so their Kirby--Siebenmann
invariants vanish.  Freedman's classification \cite{Freedman1982}
therefore gives the asserted orientation-preserving homeomorphism.
\end{proof}

Since the canonical line bundle $\mathcal K_S$ is ample, the integral
class $K_S=c_1(\mathcal K_S)$ is a K\"ahler class.  We therefore fix a
K\"ahler form $\eta$ such that
\begin{equation}\label{eq:eta-K}
 [\eta]=K_S
\end{equation}
in $H^2(S;\R)$, where $K_S$ is understood as the image of the integral
class under $H^2(S;\Z)\to H^2(S;\R)$.

\section{The counterexamples}\label{sec:construction}

\subsection{The general-type side}

Let
\[
 X^+=S\#2\CPbar^{\,2},
\]
and let $F_1,F_2$ denote the two exceptional classes.

\begin{proposition}\label{prop:plus}
There exists $m_0\ge2$ such that, for every $m\ge m_0$, the manifold
$X^+$ admits an integral symplectic form $\omega_m^+$ with
\begin{equation}\label{eq:plus-class}
 [\omega_m^+]=(3m+2)K_S-F_1-F_2.
\end{equation}
Its canonical class and numerical data are
\begin{align}
 K_{\omega_m^+}&=K_S+F_1+F_2,\label{eq:plus-K}\\
 [\omega_m^+]^2&=9m^2+12m+2,\label{eq:plus-square}\\
 K_{\omega_m^+}\cdot[\omega_m^+]&=3m+4,\label{eq:plus-pairing}\\
 e(X^+)&=13,\qquad \sigma(X^+)=-9,\label{eq:plus-topology}\\
 \ks(X^+,\omega_m^+)&=2.\label{eq:plus-kappa}
\end{align}
\end{proposition}

\begin{proof}
Choose two disjoint symplectic embeddings of standard balls of some
capacity $\varepsilon>0$ into $(S,\eta)$.
\footnote{We use the convention that the capacity of
$B^4(r)\subset(\mathbb C^2,\omega_{\mathrm{std}})$ is $\pi r^2$.
Under the blowup map $\pi\colon\widetilde X\to X$, blowing up a ball of
capacity $c$ changes the cohomology class of the symplectic form by
$-c\,\PD(E)$; equivalently,
$[\widetilde\omega]=\pi^*[\omega]-c\,\PD(E)$, where $E$ is the
exceptional sphere.}
After replacing $\eta$ by $\ell\eta$ and precomposing each embedding
with the dilation $z\mapsto z/\sqrt{\ell}$, we obtain two disjoint
balls of capacity $\ell\varepsilon$ in $(S,\ell\eta)$.  Consequently,
for all sufficiently large integers $\ell$, the symplectic manifold
$(S,\ell\eta)$ contains two disjoint balls of capacity $1$.  Blowing
up these balls produces a symplectic form with cohomology class
\[
  \ell K_S-F_1-F_2.
\]
Choose $m_0\ge2$ so that this construction is available whenever
$\ell=3m+2$ and $m\ge m_0$.  This proves existence and integrality in
\eqref{eq:plus-class}.

An almost-complex structure compatible with $\eta$ is also compatible
with $\ell\eta$, so positive rescaling does not change the canonical
class.  The canonical-class blowup formula then gives
\eqref{eq:plus-K}.  Using
$K_S^2=1$, $F_i^2=-1$, and the orthogonality of the summands, we obtain
\[
 \begin{aligned}
 [\omega_m^+]^2
 &=\bigl((3m+2)K_S-F_1-F_2\bigr)^2\\
 &=(3m+2)^2K_S^2+F_1^2+F_2^2\\
 &=(3m+2)^2-2\\
 &=9m^2+12m+2,
 \end{aligned}
\]
and
\[
 \begin{aligned}
 K_{\omega_m^+}\cdot[\omega_m^+]
 &=(K_S+F_1+F_2)\cdot\bigl((3m+2)K_S-F_1-F_2\bigr)\\
 &=3m+2+1+1=3m+4.
 \end{aligned}
\]
The topological blowup formulas and \cref{prop:CG} yield
\[
 e(X^+)=11+2=13,\qquad \sigma(X^+)=-7-2=-9.
\]
Finally, symplectic Kodaira dimension is invariant under blowup, and
$(S,\eta)$ is a minimal surface of general type.  Hence
$\ks(X^+,\omega_m^+)=2$.
\end{proof}

\subsection{The rational side}

Let
\[
 X^-=\CP^2\#10\CPbar^{\,2}.
\]
For an integer $m\ge2$, define
\begin{equation}\label{eq:ACD}
 \begin{aligned}
 A_m&=3(m^2+5m+10),\\
 C_m&=m^2+5m+9,\\
 D_m&=3m+13,
 \end{aligned}
\end{equation}
and set
\begin{equation}\label{eq:minus-class}
 \Omega_m
 =A_mH-C_m(E_1+\cdots+E_9)-D_mE_{10}.
\end{equation}

\begin{proposition}\label{prop:minus}
For every integer $m\ge2$, the class $\Omega_m$ is represented by an
integral blowup symplectic form $\omega_m^-$ on $X^-$.  Its canonical
class and numerical data are
\begin{align}
 K_{\omega_m^-}&=-3H+E_1+\cdots+E_{10},\label{eq:minus-K}\\
 [\omega_m^-]^2&=9m^2+12m+2,\label{eq:minus-square}\\
 K_{\omega_m^-}\cdot[\omega_m^-]&=3m+4,\label{eq:minus-pairing}\\
 e(X^-)&=13,\qquad \sigma(X^-)=-9,\label{eq:minus-topology}\\
 \ks(X^-,\omega_m^-)&=-\infty.\label{eq:minus-kappa}
\end{align}
\end{proposition}

\begin{proof}
First observe that
\begin{equation}\label{eq:basic-relations}
 A_m=3C_m+3,
 \qquad
 C_m-D_m=m^2+2m-4>0
\end{equation}
for every $m\ge2$.  Hence
\[
 C_m\ge\cdots\ge C_m\ge D_m>0,
 \qquad
 3C_m<A_m,
\]
so the vector
\[
 \bigl(A_m;\underbrace{C_m,\ldots,C_m}_{9\ \mathrm{times}},D_m\bigr)
\]
associated with $\Omega_m$ is ordered and reduced.  Its square in
cohomology is
\begin{align*}
 \Omega_m^2
 &=A_m^2-9C_m^2-D_m^2\\
 &=(3C_m+3)^2-9C_m^2-D_m^2\\
 &=18C_m+9-D_m^2\\
 &=9m^2+12m+2>0.
\end{align*}
By \cref{thm:KK}, there exists a blowup symplectic form
$\omega_m^-$ on $X^-$ satisfying
\[
  [\omega_m^-]=\Omega_m.
\]
Since $A_m$, $C_m$, and $D_m$ are integers, this form is integral.  This
proves the first assertion and \eqref{eq:minus-square}.

The canonical class is \eqref{eq:minus-K}.  Therefore
\begin{align*}
 K_{\omega_m^-}\cdot[\omega_m^-]
 &=-3A_m+9C_m+D_m\\
 &=-3(A_m-3C_m)+D_m\\
 &=-9+(3m+13)=3m+4,
\end{align*}
which proves \eqref{eq:minus-pairing}.  The standard connected-sum
formulas give
\[
 e(X^-)=3+10=13,\qquad \sigma(X^-)=1-10=-9.
\]
Finally, $X^-$ with a blowup form is a rational symplectic surface, so
its minimal model is $\CP^2$ and its symplectic Kodaira dimension is
$-\infty$.
\end{proof}

\subsection{Matching the invariants}

\begin{proof}[Proof of \cref{thm:main}]
Let $m\ge m_0$.  Propositions~\ref{prop:plus} and~\ref{prop:minus}
produce the required integral forms.  Their Euler characteristics and
signatures agree by \eqref{eq:plus-topology} and
\eqref{eq:minus-topology}.  Hence \eqref{eq:chern-topology} gives
\[
 c_1(\omega_m^\pm)^2=2\cdot13+3(-9)=-1,
 \qquad
 c_2(X^\pm)=13.
\]
Equations~\eqref{eq:plus-square} and~\eqref{eq:minus-square} give equal
symplectic squares.  Since $c_1(\omega)=-K_\omega$,
\eqref{eq:plus-pairing} and \eqref{eq:minus-pairing} give
\[
 c_1(\omega_m^\pm)\cdot[\omega_m^\pm]=-(3m+4).
\]
This proves all identities in \eqref{eq:main-numerics}.

Both manifolds are simply connected.  By \cref{prop:CG},
$S\homeo\CP^2\#8\CPbar^{\,2}$, so after two blowups
\[
 X^+\homeo\CP^2\#10\CPbar^{\,2}=X^-.
\]
Their Kodaira dimensions are given by \eqref{eq:plus-kappa} and
\eqref{eq:minus-kappa}.  Since symplectic Kodaira dimension is preserved
by every finite sequence of Luttinger surgeries, the two symplectic
manifolds are not Luttinger-surgery equivalent.  More strongly, the
second assertion of \cref{thm:HoLi} shows that every finite sequence
starting from $(X^-,\omega_m^-)$ remains in its symplectomorphism class.
This completes the proof.
\end{proof}

\begin{proof}[Proof of \cref{cor:Auroux-negative}]
The two manifolds in \cref{thm:main} satisfy all hypotheses of
\cref{ques:Auroux}, but its conclusion fails by the Kodaira-dimension
obstruction.
\end{proof}

\begin{corollary}\label{cor:exotic}
The smooth four-manifolds $X^-$ and $X^+$ are homeomorphic but not
diffeomorphic.
\end{corollary}

\begin{proof}
By Li's form-independence theorem \cite[Theorem~2.6]{Li2006},
symplectic Kodaira
dimension is an invariant of the oriented smooth manifold.  Hence an
orientation-preserving diffeomorphism would force the two values in
\cref{thm:main} to agree, which they do not.  An orientation-reversing
diffeomorphism would imply
\[
  \sigma(X^-)=-\sigma(X^+),
\]
but both signatures are $-9$ by \eqref{eq:plus-topology} and
\eqref{eq:minus-topology}.  Thus no diffeomorphism of either orientation
exists.
\end{proof}

\section{Nonminimality and the remaining refined questions}

On a minimal symplectic four-manifold, the signs of
$K_\omega^2=c_1(\omega)^2$ and
$K_\omega\cdot[\omega]=-c_1(\omega)\cdot[\omega]$ determine symplectic
Kodaira dimension.  Consequently, two \emph{minimal} manifolds
satisfying Auroux's numerical hypotheses automatically have the same
Kodaira dimension.  A counterexample detected only by Ho--Li's invariant
must therefore use nonminimality.  

Thus the argument does not address the following refined problem.

\begin{question}\label{ques:minimal-version}
Does Auroux's Luttinger-surgery uniqueness statement hold after requiring
both symplectic four-manifolds to be minimal?  More restrictively, what
happens if the two symplectic forms are placed on the same smooth
four-manifold?
\end{question}

For the minimal version, Auroux's four numerical invariants already
force the same symplectic Kodaira dimension, so Ho--Li's obstruction
alone cannot decide the question.  For the same-smooth-manifold version,
symplectic Kodaira dimension is independent of the symplectic form by
Li's theorem; a genuinely different Luttinger-surgery invariant would
therefore be required.

\section*{Acknowledgements}
This work was supported by JSPS KAKENHI Grant Number JP26K16990
(Grant-in-Aid for Early-Career Scientists) and by the Start-up Fund of
the Kavli Institute for the Physics and Mathematics of the Universe
(Kavli IPMU), The University of Tokyo.

The counterexample construction presented in this paper was first
suggested by ChatGPT during discussions with the author.  The author independently verified
the proposal, the cited literature, and all mathematical arguments and
computations, and takes full responsibility for the content of the
manuscript.  ChatGPT was also used for exploratory bibliographic
searches and for English and \LaTeX{} editing.

\bibliographystyle{amsplain}
\bibliography{Auroux_final}

\end{document}